\documentclass[12pt]{amsart}

\usepackage{amssymb, amsmath,amsthm}
\usepackage{upgreek}
\usepackage[english]{babel}
\usepackage[utf8]{inputenc}
\usepackage{courier}
\usepackage{enumitem}
\usepackage{caption}
\usepackage{float}

\usepackage{tikz}
\usetikzlibrary{automata,positioning,arrows}
\tikzset{->,
>=stealth,
node distance=3cm}

\newtheorem{thm}{Theorem}
\newtheorem{lem}[thm]{Lemma}
\newtheorem{cor}[thm]{Corollary}

\newtheorem{que}[thm]{Question}
\newtheorem{prop}[thm]{Proposition}
\theoremstyle{definition}

\newtheorem{ex}{Example}

\title[Cyclotomic properties of certain polynomials]{Cyclotomic properties of polynomials associated with automatic sequences}
\author{Bartosz Sobolewski}
\address{Jagiellonian University, Faculty of Mathematics and Computer Science, Institute of Mathematics, {\L}ojasiewicza 6, 30-348 Krak\'{o}w, Poland}
\email{\tt bartosz.sobolewski@doctoral.uj.edu.pl}

\begin{document}
\begin{abstract}
We show that polynomials associated with automatic sequences satisfy a certain recurrence relation when evaluated at a root of unity, which generalizes a result of Brillhart, Lomont and Morton on the Rudin--Shapiro polynomials. We study the minimal order of such a relation and the integrality of its coefficients.

\end{abstract}
\maketitle

\section{Introduction}\label{sec:Introduction}

The behavior of polynomials whose coefficients are consecutive terms of automatic sequences has been an object of interest of several authors (we recall the definition of an automatic sequence in Section \ref{sec:Preliminaries}). Probably the most widely studied examples have been the famous Rudin--Shapiro polynomials $P_n, Q_n,$ first studied by Shapiro \cite{Sh53} and Rudin \cite{Ru59} in the context of Fourier analysis. They are defined by $P_0(x) = Q_0(x)=1$ and for $n \geq 0$ the recurrence
\begin{align} \label{eq:RS_def}
P_{n+1}(x) &= P_n(x) + x^{2^n} Q_n(x),\\
Q_{n+1}(x) &= P_n(x) - x^{2^n} Q_n(x). \nonumber
\end{align}
The coefficients of $P_n$ form a $2$-automatic sequence $\{r(n)\}_{n\geq 0}$, called the Rudin--Shapiro sequence, which can be equivalently defined by $r(0)=1$ and for $n \geq 0$
$$r(2n) = r(n), \qquad r(2n+1)= (-1)^n r(n).$$

The direct motivation behind our work is the central result of the paper \cite{BLM76} by Brillhart, Lomont and Morton, who proved that the values of $P_n$ and $Q_n$ at roots of unity satisfy a certain type of recurrence relation.
More precisely, for an $r$th root of unity $\omega$ (not necessarily primitive), where $r > 1$ is odd, an integer $s \geq 2$ such that $2^s \equiv 1 \pmod{r}$ and an auxiliary sequence of polynomials $C_n\in \mathbb{Z}[x]$ (denoted $A_n$ in the original paper) the following result holds.
\begin{thm}[\cite{BLM76}, Theorem 6.1] \label{thm:RS_recur}
Let $\omega$ be an $r$th root of unity, where $r > 1$ is odd. Then for $n \geq 0$ (and $s \geq 2$)
\begin{align*}
P_{n+2s}(\omega) - C_s(\omega) P_{n+s}(\omega) + (-2)^s P_n(\omega) &= 0, \\
Q_{n+2s}(\omega) - C_s(\omega) Q_{n+s}(\omega) + (-2)^s Q_n(\omega) &= 0. \nonumber
\end{align*}
\end{thm}
In the same paper the authors also give many results concerning the integrality of the central coefficient $C_s(\omega)$. In particular, if $r$ is an odd prime power and $2$ is a primitive root modulo $r$, then it turns out that $C_s(\omega) \in \mathbb{Z}$. 

Another example is the Thue--Morse sequence, also $2$-automatic, defined by $t(0) = 1$ and for $n \geq 0$ the relations
\begin{equation} \label{eq:TM_rec}
t(2n) = t(n), \qquad t(2n+1)=-t(n).
\end{equation}
The associated Thue--Morse polynomials
$$T(n;x) = \sum_{m=0}^{n-1} t(m) x^m$$
have been considered by Doche and Mend\`{e}s France \cite{DM00}, who studied the average number of their real zeros as $n$ tends to infinity. Doche \cite{Do01} also studied generalizations of the Thue--Morse sequence in the same context. Again, it is fairly easy to show that the Thue--Morse polynomials evaluated at a root of unity of odd order satisfy a two-term recurrence relation (see Section \ref{sec:TM}).

It seems natural to ask whether or not a similar type of recurrence is satisfied at roots of unity by values of polynomials  associated with other automatic sequences. If yes, what can be said about the minimal number of coefficients in such a recurrence and the integrality of these coefficients?
In this paper we consider a general $k$-automatic sequence $\{a(n)\}_{n \geq 0}$  with values in $\mathbb{C}$ and  the polynomials
$$
A(n;x) = \sum_{m=0}^{n-1}a(m) x^m,
$$
of degree at most $n-1$.
Theorem \ref{thm:recur} of Section \ref{sec:Recurrence} answers our first question positively for a general automatic sequence and demonstrates two ways to derive a relation of the form similar as in Theorem \ref{eq:TM_rec}.
(In fact, we show that Theorem \ref{thm:RS_recur} also holds for polynomials associated with the Rudin--Shapiro sequence of degree other than $2^n-1$.) 
 In Section \ref{sec:Order} we bound the minimal number of terms in such a recurrence relation, linking it with certain properties of an automaton inducing the sequence $\{a(n)\}_{n \geq 0}$. Section \ref{sec:Integrality} is dedicated to studying the integrality of  the coefficients of the considered recurrence relation. In the final section we present the proofs of all the results in this paper.

\section{Preliminaries} \label{sec:Preliminaries}
Following \cite[Chapters 4--5]{AS03} we recall the definition and basic facts concerning deterministic finite automata with output and automatic sequences. 
Let $Q$ be a finite \emph{set of states}, $\Sigma$ the finite \emph{input alphabet}, $\delta\colon Q \times \Sigma \to Q$ the \emph{transition function}, $q_0 \in Q$ the \emph{initial state}, $\Delta$ the \emph{output alphabet} and $\uptau\colon Q \to \Delta$ the \emph{output function}.
  The sextuple $\mathcal{A} = (Q,\Sigma,\delta,q_0,\Delta,\uptau)$ is called a \emph{deterministic automaton with output} (DFAO). Denote by $\Sigma^{*}$ the set  of finite words created from letters in $\Sigma$, together with the empty word $\epsilon$. We can extend the definition of $\delta$ to $Q \times \Sigma^{*}$ by putting $\delta(q, \epsilon) = q$ for all $q \in Q$ and $\delta(q, wa) = \delta(\delta(q,w),a)$ for all $q \in Q, w \in \Sigma^{*}$ and $a \in \Sigma$. In other words, the transition function reads the input letter by letter, starting from the left. A state $q \in Q$ is called \emph{accessible} if there exists a word $w \in \Sigma^{*}$ such that $\delta(q_0, w)=q$. The automaton $\mathcal{A}$ defines a \emph{finite-state function} $f\colon \Sigma^{*} \to \Delta$ by $f(w) = \uptau( \delta(q_0, w))$. For any word $w= w_1 \cdots w_l \in \Sigma^{*}$  denote $w^R = w_l \cdots w_1 $. It can be showed that if $f$ is a finite state function, then $f^{R}\colon \Sigma^{*} \to \Delta$ defined by $f^R(w) = f(w^R)$ is also a finite state function, i.e., there exists a DFAO $\mathcal{A}' = (Q',\Sigma,\delta',q_0',\Delta,\uptau')$ such that $f(w) = \uptau( \delta'(q_0', w^R))$ (see \cite[Theorem 4.3.3]{AS03} for a constructive proof). In this situation we will say that $\mathcal{A}'$ reads the input from the right (backward). We always assume that $\mathcal{A}$ reads the input from the left (forward), unless stated otherwise.

Let $k \geq 2$ be an integer base and let $\Sigma_k = \{0,1,\ldots,k-1\}$. We call  a DFAO with input alphabet $\Sigma_k$ a $k$-DFAO. For $w \in \Sigma_k^{*}$ we denote by $[w]_k$ the integer represented in base $k$ by $w$ (we allow leading zeros). Conversely, for any nonnegative integer $n$ we use the notation $(n)_k$ for the base-$k$ expansion of $n$ without leading zeros (we put $(0)_k = \epsilon$ for any $k$).  We say that a sequence $\{a(n)\}_{n \geq 0}$ with values in $\Delta$ is $k$-automatic if there exists a $k$-DFAO $\mathcal{A} = (Q,\Sigma_k,\delta,q_0,\Delta,\uptau)$ such that $a(n) = \uptau( \delta(q_0, w))$ for all $n \geq 0$ and $w$ with $[w]_k = n$. For the purpose of this paper we will say that $\mathcal{A}$ \emph{forward-induces} $\{a(n)\}_{n\geq 0}$. By the earlier discussion, we may also define a $k$-automatic sequence using a $k$-DFAO $\mathcal{A} = (Q,\Sigma_k,\delta,q_0,\Delta,\uptau)$  reading the input starting with the least significant digit, that is $a(n) = \uptau(\delta(q_0, w^R))$ for all $n \geq 0$ and $w$ with $[w]_k = n$. In this case we will say that $\mathcal{A}$ \emph{backward-induces} $\{a(n)\}_{n\geq 0}$. In the situation where it is irrelevant whether $\mathcal{A}$ forward- or backward-induces  $\{a(n)\}_{n\geq 0}$, we will say that $\mathcal{A}$ \emph{induces}  $\{a(n)\}_{n\geq 0}$.
It is in fact enough to assume that $a(n) = \uptau(\delta(q_0, (n)_k))$ for some $k$-DFAO $\mathcal{A}$ for the sequence to be $k$-automatic.
In this case it suffices to add to $Q$ a single state $q_0'$ and modify the transition and output functions accordingly in order to obtain a new $k$-DFAO $\mathcal{A}'$, which forward-induces $\{a(n)\}_{n \geq 0}$ (see \cite[Theorem 5.2.1]{AS03} for more details).

\section{The Thue--Morse polynomials}\label{sec:TM}
In this section we study some recurrence properties of the Thue--Morse polynomials $T(n;x)$ evaluated at roots of unity of odd order. We note that the Thue--Morse sequence is both forward- and backward-induced by the $2$-DFAO in Figure 1,
and can be equivalently defined by
$$t(n) = (-1)^{s_{2}(n)},$$
where $s_2(n)$ is the sum of digits in the binary expansion of $n$.
\begin{figure}[H]
\centering
\begin{tikzpicture}
\node[state, initial](q0){$q_0/1$};
\node[state, right of=q0](q1){$q_1/-1$};

\draw
 (q0) edge[loop above] node{0} (q0)
 (q0) edge[bend left, above] node{1} (q1)
 (q1) edge[bend left, below] node{1} (q0)
 (q1) edge[loop above] node{0} (q1);
\end{tikzpicture}
\caption{A $2$-DFAO inducing the Thue--Morse sequence}
\label{fig:TM}
\end{figure}
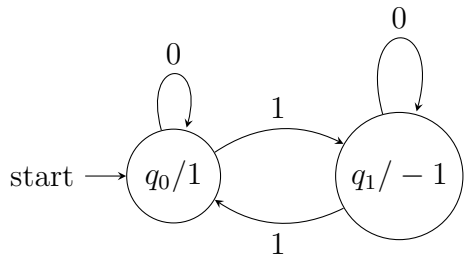
First, we recall some standard properties of the polynomials $T(n;x)$.
\begin{prop}\label{prop:TM_props}
Let $n \geq 1, s \geq 0$ be integers. Then
\begin{enumerate}[label={\textup{(\roman*)}}]
\item $T(2n;x) = (1-x) T(n;x^2);$
\item $T(2^s;x) = \prod_{i=0}^{s-1}(1-x^{2^i});$
\item $T(2^s n;x) = T(n;x^{2^s}) T(2^s;x);$
\item $x^{2^s-1} T\left(2^s; \frac{1}{x}\right) = (-1)^s T(2^s; x).$
\end{enumerate}
\end{prop}
The derivation for $T(n;x)$ of a similar type of recurrence as in Theorem \ref{thm:RS_recur} is almost immediate. More precisely, let $r \geq 1$ be odd, let $\omega$ be an $r$th root of unity (not necessarily primitive), and fix $s \geq 1$  such that $2^s \equiv 1 \pmod{r}$. By putting $x=\omega$ in identity (iii) of Proposition \ref{prop:TM_props}, we obtain the following result.
\begin{prop} \label{prop:TM_rec}
For all integers $n\geq 1$ we have
$$T(2^s n;\omega) = T(2^s;\omega)  T(n;\omega).$$
\end{prop}
A more difficult question concerns the integrality of the coefficient $T(2^s;\omega)$ and computation of its value. We point out that this number is an algebraic integer, hence  $T(2^s;\omega) \in \mathbb{Z}$ if and only if $T(2^s;\omega) \in \mathbb{Q}$.

  Let $r_0\geq 1$ be minimal such that $\omega^{r_0} = 1$ and let $s_0$ denote the multiplicative order of $2$ modulo $r_0$. Clearly, $r$ and $s$ are multiples of $r_0$ and $s_0$, respectively. We can thus restrict ourselves to studying the value $T(2^{s_0};\omega)$, since for any positive integer $m$ we have
$$T(2^{ms_0};\omega) = (T(2^{s_0};\omega))^m.$$
We leave out the trivial case $r_0=1$ from the following considerations, as then $\omega=1, s_0=1$ and $T(2;1)=0$. Let $\varphi$ be the Euler totient function and let $\psi_2 \in \operatorname{Gal}(\mathbb{Q}(\omega)/\mathbb{Q})$ denote the Galois automorphism of $\mathbb{Q}(\omega)$ taking $\omega$ to $\omega^2$. We start with a general observation.

\begin{prop} \label{prop:real}
If $s_0$ is even, then $T(2^{s_0};\omega)$ is real; otherwise, $T(2^{s_0};\omega)$ is purely imaginary. In both cases $T(2^{s_0};\omega)$ lies in the subfield of $\mathbb{Q}(\omega)$ fixed by the subgroup of $\operatorname{Gal}(\mathbb{Q}(\omega)/\mathbb{Q})$ generated by $\psi_2$.
\end{prop}

Using this observation, we can determine the possible values of $T(2^{s_0};\omega)$ when $r_0$ is an odd prime power and either $2$ is a primitive root modulo $r_0$ or $s_0 = \varphi(r_0)/2$ is odd.
\begin{prop} \label{prop:prime_power}
Assume that $r_0 = p^{\alpha}$, where $p$ is an odd prime number. We have the following:
\begin{enumerate}[label={\textup{(\roman*)}}]
\item$T(2^{s_0},\omega) \in \mathbb{Z}$ if and only if $s_0=\varphi(r_0)$. In this case $T(2^{s_0},\omega) = p$;
\item If $s_0 = \varphi(r_0)/2$ is odd, then $\{T(2^{s_0};\omega),  T(2^{s_0};\omega^{-1}) \} = \{i \sqrt{p}, -i \sqrt{p}\}.$
\end{enumerate}
\end{prop}

We now turn our attention to the situation when $r_0$ has two or more distinct prime factors. In the following proposition we give the only possible integral values of $T(2^{s_0};\omega)$ and exhibit two special cases when it is possible to determine whether $T(2^{s_0};\omega)$ is an integer.
\begin{prop} \label{prop:distinct_factors}
Assume that $r_0$ is odd and has at least two distinct  prime factors. We have the following:
\begin{enumerate}[label={\textup{(\roman*)}}]
 \item If $T(2^{s_0};\omega) \in \mathbb{Z}$, then $T(2^{s_0};\omega) \in \{1,-1\}$;
 \item If $s_0 = \varphi(r_0)/2$ and $2^{s_0/2} \not\equiv -1 \pmod{r_0}$, then $T(2^{s_0};\omega) \in \{1,-1\}$;
 \item If $s_0$ is even and $2^{s_0/2} \equiv -1 \pmod{r_0}$, then $T(2^{s_0};\omega) \in \mathbb{R} \setminus \mathbb{Z}$.
 \end{enumerate}
\end{prop}

Direct computation shows that the converse implication in Proposition \ref{prop:distinct_factors}(ii) does not hold.
We have performed numerical calculations of $T(2^{s_0};\omega)$  in Mathematica for odd $r_0 \in [3,10^5]$ having at least two distinct prime factors and $\omega = \exp(2 \pi i/r_0)$. The choice of $\omega$ for each fixed $r_0$ may only affect the value $T(2^{s_0};\omega)$ if this coefficient is not integral, which is irrelevant in the following discussion. 
Our aim has been to investigate how often $T(2^{s_0};\omega) = 1$ and $T(2^{s_0};\omega) = -1$, as this distinction does not follow from the results above. By Proposition \ref{prop:real} and Proposition \ref{prop:distinct_factors}(iii) we can restrict our attention to the set $R$ of $r_0$ such that $s_0$ is even and $2^{s_0/2} \not\equiv -1 \pmod{r_0}$. This condition holds for $32921$ out of all $40315$ considered values $r_0$. We further partition $R$ depending on whether $T(2^{s_0};\omega) = 1, T(2^{s_0};\omega) = -1$ or $T(2^{s_0};\omega) \not\in \mathbb{Z}$ as well as whether $\varphi(r_0)=2s_0$ or $\varphi(r_0) > 2s_0$.
In Table \ref{tab:num} below we give the cardinality of each such subset.
\begin{table}[H] 
\begin{center}
\begin{tabular}{l|ll|l} 
&$\varphi(r_0)=2s_0$ & $\varphi(r_0) > 2s_0$ & Total   \\ \hline
$T(2^{s_0};\omega) = 1$  & 2728  &   1143  & 3871   \\
$T(2^{s_0};\omega) = -1$  & 2935  &    1481 & 4416   \\
$T(2^{s_0};\omega) \not\in \mathbb{Z}$  & 0 & 24634 & 24634 \end{tabular}
\end{center}
\caption{Numerical results concerning the integrality of $T(2^{s_0};\omega)$}
\label{tab:num}
\end{table}
The case $\varphi(r_0)=2s_0$ corresponds to Proposition \ref{prop:distinct_factors}(ii), thus $T(2^{s_0};\omega) \in \{1,-1\}$ for all such $r_0$. We see that the value $1$ is attained in about $48.2\%$ of cases. However, we have not been able to identify a general rule determining whether $T(2^{s_0};\omega) = 1$ or $T(2^{s_0};\omega) = -1$. The second case $\varphi(r_0)>2s_0$ reveals some ``unexpected'', though rare, occurences of $T(2^{s_0};\omega) \in \{1,-1\}$, which are not covered by Proposition \ref{prop:distinct_factors}(ii). More precisely, we have $T(2^{s_0};\omega) = 1$ approximately $4.2\%$ of the time, while  $T(2^{s_0};\omega) = -1$ occurs about $5.4\%$ of the time.

\section{The polynomial matrix associated with an automaton} \label{sec:Matrix}
In this section we start working towards obtaining a recurrence relation (formulated precisely in Section \ref{sec:Recurrence}) involving polynomials associated with an arbitrary automatic sequence. 
Consider a $k$-DFAO $\mathcal{A} = (Q,\Sigma_k,\delta,q_0,\Delta,\uptau)$ with $Q = \{q_0, q_1, \ldots, q_{d-1}\}$  and $\Delta \subset \mathbb{C}$.
We associate with $\mathcal{A}$ a $d \times d$ matrix $M(x) = [m_{ij}(x)]_{0 \leq i,j \leq d-1}$ with polynomial entries, where
$$m_{ij}(x) = \sum_{\underset{ \delta(q_i , a) = q_j}{a \in \Sigma_k}} x^a.$$
The matrix $M(x)$ carries all the relevant information concerning the transitions between states in $\mathcal{A}$ and does not depend on the output.
Note that $M(1)$ is the transpose of the incidence matrix associated with $\mathcal{A}$ (cf. \cite[Section 8.2]{AS03}). 
We also denote for $t \geq 0$
\begin{align*}
M(k^t;x) &= [m_{ij}(k^t;x)]_{0 \leq i,j \leq d-1} = M(x^{k^{t-1}})  M(x^{k^{t-2}}) \cdots  M(x^{k})  M(x), \\
M^{R}(k^t;x) &= [m^R_{ij}(k^t;x)]_{0 \leq i,j \leq d-1} = M(x) M(x^{k})\cdots  M(x^{k^{t-2}}) M(x^{k^{t-1}}).
\end{align*}
In particular, $M(k;x)= M^R(k;x)=M(x)$. These matrices will play an important role in the results of Section \ref{sec:Recurrence}. Below we establish some of their basic properties.
\begin{prop} \label{prop:matrix_coef}
For all $t \geq 0$ we have
\begin{align}
m_{ij}(k^t;x) &= \sum_{\underset{\delta(q_i , w)=q_j}{w \in \Sigma_k^t}} x^{[w]_k}, \label{eq:matrix_coef_1}\\ 
m^R_{ij}(k^t;x) &= \sum_{\underset{\delta(q_i , w)=q_j}{w \in \Sigma_k^t}} x^{[w^R]_k}.  \label{eq:matrix_coef_2}
\end{align}
\end{prop}
Roughly speaking, both $m_{ij}(k^t;x)$ and $m^R_{ij}(k^t;x)$  describe all paths of length $t$ between the states $q_i$ and $q_j$.
We may also look at $M(k^t;x)$ and $M^R(k^t;x)$ as polynomials in $x$ with matrix coefficients. For $w \in \Sigma_k^{*}$ let $M_w$ be the $d \times d$ integer matrix whose entry $m_{w,ij}$ at position $(i,j) \in \{0,1,\ldots,d-1\}^2$ is equal to $1$ if $\delta(q_i , w) = q_j$ and $0$ otherwise. In particular, $M_\epsilon$ is the $d \times d$ identity matrix.
 The following  observation provides an alternative description of $M(k^t;x)$ and $M^R(k^t;x)$, where the matrices $M_w$ play the role of the coefficients of an appropriate polynomial.
\begin{prop} \label{prop:matrix_poly}
For all $t \geq 0$ we have
\begin{align}
M(k^t;x) = \sum_{w \in \Sigma_k^t}  M_w x^{[w]_k}, \label{eq:matrix_poly_1} \\
M^R(k^t;x) = \sum_{w \in \Sigma_k^t}M_{w} x^{[w^R]_k}. \label{eq:matrix_poly_2}
\end{align}
\end{prop}

We also define for integers $n \geq 1$ and $t$ such that $k^{t-1}+1 \leq n \leq k^t$, the matrices
$$M(n;x) = \sum_{\underset{[w]_k\leq n-1}{w \in \Sigma_k^t}}  M_w x^{[w]_k}$$
to be the truncation of $M(k^t;x)$, viewed as a polynomial in $x$ (in the case of $M^R$ this will not be needed).

As we will see in the following sections, a recurrence relation of the desired form derived using $M(x)$ may not have minimal order. In some cases it is possible to construct another square matrix $\widehat{M}(x)$, which leads to a better result  in the above sense.
First, we associate with each  $q_i \in Q$ a finite state function $f_i \colon \Sigma_k^{*} \rightarrow \mathbb{C}$ given by
\begin{equation*} 
 f_i(w) = \uptau(\delta(q_i,w)).
\end{equation*}
We can think of $f_i(w)$ as the output that $\mathcal{A}$ would return at input $w \in  \Sigma_k^{*}$, if the initial state were changed to $q_i$.
Assume, after renumbering the states in $Q \setminus \{q_0\}$, that the set  $\{f_0, f_1, \ldots, f_c\}$ generates $\operatorname{span}_{\mathbb{C}} \{f_0, f_1, \ldots, f_{d-1}\}$, where $0 \leq c \leq d-1$. We stress that $f_0, f_1, \ldots, f_c$ do not have to be linearly independent. If $c=d-1$, we put $\widehat{M}(x) = M(x)$. Otherwise, we construct $\widehat{M}(x)$ as described below. 
For any integer $p$ such that $c < p \leq d-1$ write
\begin{equation} \label{eq:linear_dep}
f_p = \sum_{j=0}^{c} \alpha_{pj} f_j,
\end{equation}
 where $\alpha_{p0}, \ldots, \alpha_{pc} \in \mathbb{C}$.
We can now define $\widehat{M}(x) = [\widehat{m}_{ij}(x)]_{0 \leq i,j \leq c}$ of size $(c+1)\times (c+1)$, with entries
$$\widehat{m}_{ij}(x) = m_{ij}(x) +  \sum_{p =c+1}^{d-1} \alpha_{pj}m_{ip}(x).$$
In other words, to construct $\widehat{M}(x)$ we add for each $i = 0,1, \ldots,c$ the vector $ \sum_{p =c+1}^{d-1} m_{ip}(x) [\alpha_{p0}, \ldots, \alpha_{pc}]$ to the $i$th row of of $M(x)$  and delete the rows and columns with indices greater than $c$.

Observe that, unlike $M(x)$, the matrix $\widehat{M}(x)$ (its dimension, in particular) depends on the output of the $k$-DFAO considered.
The above construction raises the problem of finding linear dependence relations among the $f_i$, if any exist.
We discuss a possible solution and further implications in Section \ref{sec:Order} below.

Similarly as with $M(x)$, we may define for $t \geq 0$ the matrices
\begin{align*}
\widehat{M}(k^t;x) &=  \widehat{M}(x^{k^{t-1}})  \widehat{M}(x^{k^{t-2}}) \cdots  \widehat{M}(x^{k})  \widehat{M}(x), \\
\widehat{M}^{R}(k^t;x) &=  \widehat{M}(x) \widehat{M}(x^{k})\cdots  \widehat{M}(x^{k^{t-2}}) \widehat{M}(x^{k^{t-1}}).
\end{align*}
Take integers $n \geq 1$ and $t$ such that $k^{t-1}+1 \leq n \leq k^t$ and write
\[ \widehat{M}(k^t;x) = \sum_{m=0}^{k^t-1} \widehat{M}_m x^m,\]
where each $\widehat{M}_m$ is a square matrix of dimension $c+1$. Then we define 
\[ \widehat{M}(n;x) = \sum_{m=0}^{n-1} \widehat{M}_m x^m. \]

We illustrate the construction of $M(x)$ and $\widehat{M}(x)$ with two simple examples.
\begin{ex} \label{ex:RS1}
The Rudin--Shapiro sequence is both forward- and backward-induced by the $2$-DFAO displayed in Figure \ref{fig:RS}.
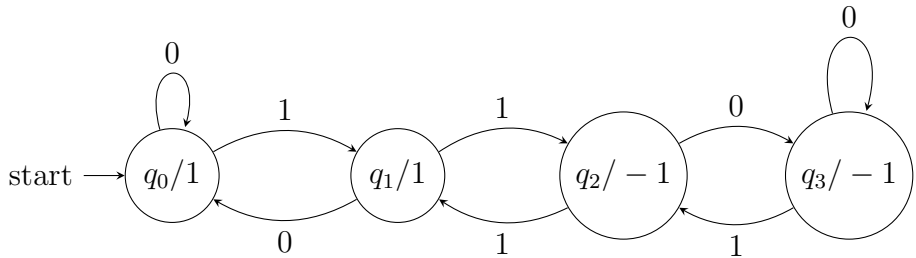
\begin{figure}[H]
\centering
\begin{tikzpicture}
\node[state, initial](q0){$q_0/1$};
\node[state, right of=q0](q1){$q_1/1$};
\node[state, right of=q1](q2){$q_2/-1$};
\node[state, right of=q2](q3){$q_3/-1$};

\draw
 (q0) edge[loop above] node{0} (q0)
 (q0) edge[bend left, above] node{1} (q1)
 (q1) edge[bend left, below] node{0} (q0)
 (q1) edge[bend left, above] node{1} (q2)
 (q2) edge[bend left, above] node{0} (q3)
 (q2) edge[bend left, below] node{1} (q1)
 (q3) edge[bend left, below] node{1} (q2)
 (q3) edge[loop above] node{0} (q3);

\end{tikzpicture}
\caption{A $2$-DFAO inducing the Rudin--Shapiro sequence}
\label{fig:RS}
\end{figure}
The associated polynomial matrix is
\[ M(x) = \begin{bmatrix}
1 & x & 0 & 0 \\
1 & 0 & x & 0 \\
0 & x & 0 & 1 \\
0 & 0 & x & 1 \\
\end{bmatrix}. \]
As in the previous example, it is not hard to check that $f_2=-f_1$ and $f_3=-f_0$. The choice $c=1$ yields the matrix
\[\widehat{M}(x) = \begin{bmatrix}
1 & x  \\
1 & -x 
\end{bmatrix}. \]
Not coincidentally, the relation \eqref{eq:RS_def} defining the Rudin--Shapiro polynomials can be written as
\[
\begin{bmatrix} P_{n+1}(x) \\  Q_{n+1}(x) \end{bmatrix} = \widehat{M}(x^{2^n}) \begin{bmatrix}  P_n(x)  \\ Q_n(x) \end{bmatrix}.
\]
A similar type of a polynomial recurrence relation is derived in Proposition \ref{prop:poly_recur} of Section \ref{sec:Recurrence} for the polynomials associated with any automatic sequence.
\end{ex}

\begin{ex} \label{ex:BS1}
Let $\{b(n)\}_{n \geq 0}$ be the Baum--Sweet sequence, given by $b(n) = 1$ if $(n)_2$ contains no block of zeros of odd length, and $b(n) = 0$ otherwise. This is a $2$-automatic sequence backward-induced by the $2$-DFAO displayed in Figure \ref{fig:BS}.
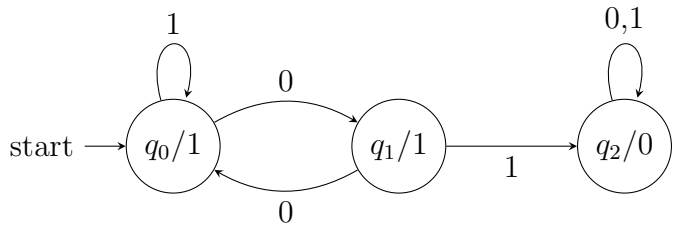
\begin{figure}[H]
\centering
\begin{tikzpicture}
\node[state, initial](q0){$q_0/1$};
\node[state, right of=q0](q1){$q_1/1$};
\node[state, right of=q1](q2){$q_2/0$};

\draw
 (q0) edge[loop above] node{1} (q0)
 (q0) edge[bend left, above] node{0} (q1)
 (q1) edge[bend left, below] node{0} (q0)
 (q1) edge[below] node{1} (q2)
 (q2) edge[loop above] node{0,1} (q2);

\end{tikzpicture}
\caption{A $2$-DFAO backward-inducing the Baum--Sweet sequence}
\label{fig:BS}
\end{figure}
The associated polynomial matrix is
$$ 
M(x) = \begin{bmatrix}
x & 1 & 0  \\
1 & 0 & x  \\
0 & 0 & 1+x \\
\end{bmatrix}.
$$
However, the finite-state function $f_2$ is constantly $0$, thus we can consider the matrix
$$ 
\widehat{M}(x) = \begin{bmatrix}
x & 1   \\
1 & 0   \\
\end{bmatrix},
$$
corresponding to $f_0, f_1$.
\end{ex}

\section{Recurrence relations at roots of unity}\label{sec:Recurrence}

Let $\{a(n)\}_{n \geq 0}$ be a $k$-automatic sequence with values in $\mathbb{C}$ and define the polynomials
$$ A(n;x) = \sum_{m=0}^{n-1}a(m) x^m, $$
similarly as in Section \ref{sec:Introduction}.
Let $\omega$ be an $r$th root of unity (not necessarily primitive) with $r$ and $k$ relatively prime and let $s \geq 1$ be an integer such that $k^s \equiv 1 \pmod{r}$.
The main goal of this section is to establish a recurrence relation, valid for all $n \geq 1$, which is of the form
\begin{equation} \label{eq:recurrence_form}
\sum_{m=0}^{l}C_m(\omega) A(k^{ms}n;\omega) = 0, 
\end{equation}
where $C_0, \ldots, C_l \in \mathbb{C}[x]$ depend only on $s$, with $C_l(x)$ nonzero.

Let $\mathcal{A} = (Q,\Sigma_k,\delta,q_0,\Delta,\uptau)$ be a $k$-DFAO inducing $\{a(n)\}_{n \geq 0}$. 
We retain the notation introduced in Section \ref{sec:Matrix}.
To begin with, we show a general polynomial recurrence relation involving $A(n;x)$, whose form depends on whether $\mathcal{A}$ computes $\{a(n)\}_{n \geq 0}$ by reading the input starting from the most or the least significant digit. 
With each state $q_i \in Q$ we associate two sequences of polynomials $F_i(n;x),  F_i^R(n;x)$, where for positive integers $n \geq 1$ and $t$ such that $k^{t-1}+1 \leq n \leq k^{t}$, we define
\begin{align*}
 F_i(n;x) &= \sum_{\underset{ [w]_k \leq  n-1}{w \in \Sigma_k^{t}}} f_i(w) x^{[w]_k}, \\
 F_i^R(n;x) &= \sum_{\underset{ [w]_k \leq  n-1}{w \in \Sigma_k^{t}}} f_i(w^R) x^{[w]_k}.
\end{align*}
It is clear that if $\mathcal{A}$ forward-induces   $\{a(n)\}_{n \geq 0}$, then $F_0(n;x) = A(n;x)$. Similarly if $\mathcal{A}$ backward-induces  $\{a(n)\}_{n \geq 0}$, then $F_0^R(n;x) = A(n;x)$.

Assume that the set $\{f_0, f_1, \ldots, f_c\}$ generates $\operatorname{span}_{\mathbb{C}} \{f_0, f_1, \ldots, f_{d-1}\}$ and the matrix $\widehat{M}(x)$ corresponds to this set.
Consider the column vectors
\begin{align*}
\widehat{F}(n;x) &= \begin{bmatrix} F_0(n;x) \cdots F_{c}(n;x) \end{bmatrix}^T, \\
\widehat{F}^R(n;x) &= \begin{bmatrix} F^R_0(n;x) \cdots F^R_{c}(n;x) \end{bmatrix}^T.
\end{align*}
We have the following polynomial recurrence relations.

\begin{prop} \label{prop:poly_recur}
For all integers $u \geq 1$ and $n \geq 1$ we have
\begin{align}
\widehat{F}(k^un;x) &=   \widehat{M}(n; x^{k^u}) \widehat{F}(k^u;x), \label{eq:recur_left} \\
\widehat{F}^R(k^un;x) &= \widehat{M}^R(k^u;x) \widehat{F}^R(n;x^{k^u})\label{eq:recur_right}.
\end{align}
\end{prop}
 Note that in the expressions on the right hand side of \eqref{eq:recur_left} and \eqref{eq:recur_right}, the pairs of arguments $(n;x^{k^u})$ and $(k^u;x)$ essentially switch roles. 
%


Now, fix $r \geq 1$ and let $\omega$ be any $r$th root of unity. Take $s \geq 1$  such that $k^s \equiv 1 \pmod{r}$.
We are ready to state the main result in this section.

\begin{thm} \label{thm:recur}
Let $C(x,y) \in \mathbb{C}[x,y]$ and write
$$C(x,y) = \sum_{m=0}^{l} C_m(x) y^m.$$
Assume that one of the following conditions holds:
\begin{enumerate}[label={\textup{\arabic*.}}]
\item $\mathcal{A}$ forward-induces $\{a(n)\}_{n \geq 0}$ and $C(\omega, \widehat{M}(k^s;\omega)) = 0,$
\item $\mathcal{A}$ backward-induces $\{a(n)\}_{n \geq 0}$ and $C(\omega, \widehat{M}^R(k^s;\omega)) = 0.$
\end{enumerate}
Then for all $n \geq 1$ we have
$$\sum_{m=0}^{l} C_m(\omega) A(k^{ms}n;\omega) = 0.$$
\end{thm}

Observe that in Theorem \ref{thm:recur} we can always choose a polynomial $C(x,y)$ which only depends on $\omega$ through $s$, in the following sense: for $s \geq 1$  fixed, one of the conditions in Theorem \ref{thm:recur} is satisfied for all $r$th roots of unity $\omega$, whenever $r$ divides $k^s-1$.  Indeed, if $\mathcal{A}$ forward-induces $\{a(n)\}_{n \geq 0}$, then it suffices to take $C(x,y)$ to be the characteristic polynomial of $\widehat{M}(k^s;x)$ over the field of rational functions $\mathbb{C}(x)$. Similarly if $\mathcal{A}$ backward-induces $\{a(n)\}_{n \geq 0}$, then one can take the characteristic polynomial of $\widehat{M}^R(k^s;x)$. We illustrate Theorem \ref{thm:recur} by continuing the two examples of the previous section.

\begin{ex} \label{ex:RS2}
We consider the polynomials $R(n; x) = \sum_{m=0}^{n-1} r(m)x^m$, corresponding to the Rudin--Shapiro sequence. If $n=2^u$, then $R(n;x)$ coincides with the Rudin--Shapiro polynomial $P_u(x)$.

Let $\omega^3=1$ and $s=2$ so that $2^s \equiv 1 \pmod{3}$. To obtain a recurrence relation involving $R(n;\omega)$, we will use the matrix $\widehat{M}(x)$ already constructed in Example \ref{ex:RS1}.
 The characteristic (and minimal) polynomial of $\widehat{M}(2^2;x)$ is
 \begin{align*}
 C(x,\lambda) = \lambda^2 -(x^3+x^2+x+1) \lambda + 4x^3,
 \end{align*}
 so $C(1,\lambda) = \lambda^2 -4\lambda+4$ and $C(\omega,\lambda) = \lambda^2 - \lambda^+4$, when $\omega$ is a primitive $3$rd root of unity. Theorem \ref{thm:recur} gives for all $n \geq 1$ the relations
 \begin{align*}
 R(2^4n;1)-4R(2^2n;1)+4R(n;1)&=0,  \nonumber \\
 R(2^4n;\omega)-R(2^2n;\omega)+4R(n;\omega)&=0.  
\end{align*}
We can similarly derive a recurrence relation involving an $r$th root of unity $\omega$ for any odd $r$ and appropriate $s$. It is straigforward to check that $\det (\widehat{M}^R(2^s;x)) = (-2)^s x^{2^s-1}$, thus by considering the characteristic polynomial of $\widehat{M}^R(k^s;x)$ we obtain a recurrence relation of the form
\[R(2^{2s}n;\omega) - C_s(\omega) R(2^sn;\omega)+(-2)^s R(n;\omega) = 0, \]
where $C_s(x) = \operatorname{tr} (\widehat{M}(2^s;x))$ is a polynomial with integer coefficients.
This improves the result Theorem \ref{thm:RS_recur} in the sense that the obtained recurrence relation works also for $n$ other than powers of $2$.  

A similar procedure, starting with the matrix $M(x)$, only yields a $5$-term recurrence relation. On the other hand, it holds regardless of the output of the $2$-DFAO.
%
\end{ex}

\begin{ex} \label{ex:BS2}
Consider the polynomials  $B(n; x) = \sum_{m=0}^{n-1} b(m)x^m$, associated with the Baum--Sweet sequence.
 Using the matrix $\widehat{M}(x)$ constructed in Example \ref{ex:BS1}, we find that for all $n \geq 1$
\[B(2^{2s}n;\omega) - C_s(\omega) B(2^sn;\omega)+(-1)^s B(n;\omega) = 0, \]
where $C_s(x) = \operatorname{tr} (\widehat{M}^R(2^s;x))$.
\end{ex}

\section{The order of the recurrence} \label{sec:Order}
The examples in the previous section show that the minimal order $l$ of the recurrence relation of the form \eqref{eq:recurrence_form} can be bounded from above by the dimension of $\operatorname{span}_{\mathbb{C}}\{f_0, \ldots, f_{d-1}\}$, which in turn is at most the number of states in a $k$-DFAO inducing the sequence $\{a(n)\}_{n\geq 0}$. In this section we make these observations more precise and discuss a method to find linear dependence relations among $f_0, \ldots, f_{d-1}$. We apply the results to a certain class of pattern counting sequences. Before stating any results we give a simple example, which demonstrates that $l$ might depend on the choice of $r$ and $\omega$.

\begin{ex}\label{ex:TM2}
Consider a variant of the Thue--Morse sequence
$$\widetilde{t}(n) = s_2(n) \bmod{2},$$
which is the image of $\{t(n)\}_{n \geq 0}$ under the coding $1 \mapsto 0, -1 \mapsto 1$. Define the polynomials
$$\widetilde{T}(n;x) = \sum_{m=0}^{n-1} \widetilde{t}(m) x^m.$$
Clearly, the values $\widetilde{T}(n;\omega)$ satisfy a three-term recurrence relation of the form \eqref{eq:recurrence_form}, and it is not immediately clear whether or not the number of terms can be reduced, like in the case of usual Thue--Morse polynomials. Observe that $2\widetilde{t}(n) = 1- t(n)$ for any $n \geq 1$, thus
$$2 \widetilde{T}(n;x) = \frac{x^n-1}{x-1} - T(n;x).$$
Denote $C = T(2^s;\omega)$ for simplicity of notation. We have for all $n \geq 1$
\begin{align*}
\widetilde{T}(2^{s} n; \omega) &= \frac{1}{2}\left(  \frac{\omega^n-1}{\omega-1} - T(2^{s} n; \omega)\right) = \frac{1}{2}\left(  \frac{\omega^n-1}{\omega-1} - C T( n; \omega)\right) \\
&= \frac{1}{2}\left(  \frac{\omega^n-1}{\omega-1} - C \left(\frac{\omega^n-1}{\omega-1} - 2\widetilde{T}(n;\omega) \right)\right) \\ &= C \widetilde{T}(n;\omega) +\frac{(1-C)(\omega^n-1)}{2(\omega-1)}.
\end{align*}
Therefore, if there exists a two-term recurrence relation of the form
$$\widetilde{T}(2^{s} n; \omega) = \widetilde{C} \widetilde{T}(n;\omega),$$
valid for all $n \geq 1$, then we must have $\widetilde{C}=C=1$. Conversely, if $C=1$, then $\widetilde{T}(2^{s} n; \omega) =  \widetilde{T}(n;\omega)$ for all $n \geq 1$. We have seen in Section \ref{sec:TM} that in this case $r$ must have at least two distinct prime factors, but it is hard to determine in general for which $r$ we have $T(2^s;\omega)=1$.
\end{ex}

Example \ref{ex:TM2} discourages us from considering the minimal number of terms in \eqref{eq:recurrence_form} for each $r$ and $\omega$ individually. Indeed, even for the fairly simple sequence $\{\widetilde{t}(n)\}_{n\geq 0}$ it seems very difficult to determine the answer without direct calculation.

Therefore, for a $k$-automatic sequence $\{a(n)\}_{n \geq 0}$, it seems reasonable to consider a global bound on the minimal number of terms in the recurrence \eqref{eq:recurrence_form}, independent of $r,s$ and $\omega$. More precisely, we consider the minimal $l \geq 1$ such that for each $r \geq 1, r$th root of unity $\omega$ and $s \geq 1$ such that $k^s \equiv 1 \pmod{r}$ there exists a recurrence relation of the form  \eqref{eq:recurrence_form} satisfied for all $n \geq 1$, where $C_0, \ldots, C_l \in \mathbb{C}[x]$ depend only on $s$, with $C_l$ nonzero. We denote this number, which depends solely on the sequence $\{a(n)\}_{n \geq 0}$,  by $l_{\min}$. Our aim is to obtain a bound on $l_{\min}$, relying on the properties of an automaton inducing $\{a(n)\}_{n \geq 0}$.
As we have already observed, such a bound can immediately be obtained from Theorem \ref{thm:recur}.  More precisely, we have the following result.
\begin{prop} \label{prop:simple_bound}
Assume that the sequence $\{a(n)\}_{n \geq 0}$ is (forward- or backward-) induced by a $k$-DFAO $\mathcal{A}$ with $d$ states. Then 
\[l_{\min} \leq \dim(\operatorname{span}_{\mathbb{C}}\{f_0, \ldots, f_{d-1}\}).\]
In particular, $l_{\min} \leq d$.
\end{prop}
In order to achieve the best possible bound, we may first remove all inaccessible states from $\mathcal{A}$. Note that the bound $l_{\min} \leq d$, while in general not optimal, is applicable regardless of the output of $\mathcal{A}$. 
Unfortunately, we have neither been able to find an example where the inequality of Proposition \ref{prop:simple_bound} is sharp, nor managed to prove that equality holds. We therefore ask the following question.
\begin{que} \label{que:min_order}
Assume that $\mathcal{A}$ has no inaccessible states. Does the equality 
\[l_{\min} = \dim(\operatorname{span}_{\mathbb{C}}\{f_0, \ldots, f_{d-1}\})\]
hold?
\end{que}
In its statement question deliberately did not specify whether $\mathcal{A}$ forward- or backward-induces $\{a(n)\}_{n \geq 0}$. In fact, we believe that such an assumption does not affect the answer, since the construction described in Section \ref{sec:Matrix} and the result of Theorem \ref{thm:recur} are almost identical in both cases. This raises another question, which seems interesting in its own right. 
\begin{que}\label{que:dim_finite_state}
Let $f_0, \ldots, f_{d-1}$ be the finite-state functions corresponding to the accessible states of $\mathcal{A}$. Let $\mathcal{B}$ be a $k$-DFAO such that  the finite-state function $g_0 = f_0^R$ corresponds to its initial state. Let $g_0, \ldots, g_{e-1}$ denote the finite-state functions corresponding to the accessible states of $\mathcal{B}$. Does the equality
\[\dim(\operatorname{span}_{\mathbb{C}}\{f_0, \ldots, f_{d-1}\}) = \dim(\operatorname{span}_{\mathbb{C}}\{g_0, \ldots, g_{e-1}\})\]
hold?
\end{que}
Observe that an affirmative answer to Question \ref{que:min_order} would immediately give an affirmative answer to Question \ref{que:dim_finite_state} in the case when the considered automata forward- and backward-induce $\{a(n)\}_{n \geq 0}$, respectively. 

We will now give a straightforward approach to determine linear dependence relations among the finite-state functions $f_0, \ldots,
 f_{d-1}$. Let $S \subset Q^d$ be the set of all distinct $d$-tuples $(\delta(q_0,w),\ldots,\delta(q_{d-1},w))$ with $w \in \Sigma_k^{*}$. In other words, this is the set of states in the $d$-fold product of $\mathcal{A}$ with itself (in the sense of \cite[p. 22]{Ko97}) that are accessible from $(q_0,\ldots,q_{d-1})$. We can find them using the following simple algorithm. Define the function $\delta^d \colon Q^d \times \Sigma_k \rightarrow Q^d$ by
 \[ \delta^d((q_{j_0}, \ldots, q_{j_{d-1}}), a) = (\delta(q_{j_0},a) \ldots, \delta(q_{j_{d-1}},a)). \]
 
Put  $S_0 = \{(q_0,\ldots,q_d)\}$ and $S_{i+1} = \delta^d(S_i \times \Sigma_k) \setminus \bigcup_{j=0}^{i} S_j$ for $i \geq 0$.   The set $S_i$ contains precisely the $d$-tuples of the form $(\delta(q_0,w),\ldots,\delta(q_{d-1},w))$ with $|w|=i$, that do not belong in any $S_j$ with $j<i$. It is clear that the $S_i$ are pairwise disjoint and that there exists $i_0 \geq 0$ such that $S_i$ is empty for $i > i_0$, and otherwise nonempty. We have
\[ S =  \bigcup_{i=0}^{i_0} S_i.\]
The number $i_0$ can be roughly estimated from above by $d^d$.
 
Choose $w_1, \ldots, w_e \in \Sigma_k^{*}$ such that the $d$-tuples $(\delta(q_0,w_j),\ldots,\delta(q_{d-1},w_j))$ are all distinct and form the whole set $S$.  The following proposition asserts that it is necessary and sufficient for our purpose to consider linear dependence relations among the functions $f_i$ restricted to the set $\{w_1,\ldots,w_e\}$ (each such restriction can be considered a vector in $\mathbb{C}^{e}$).
 
\begin{prop} \label{prop:lin_dep}
We have 
\[\sum_{i=0}^{d-1} \beta_i f_i = 0\]
for some $\beta_0, \ldots, \beta_{d-1} \in \mathbb{C}$ if and only if for all $j=1,\ldots,e$
\[ \sum_{i=0}^{d-1} \beta_i f_i(w_j) = 0.\]
\end{prop}
 
Although Proposition \ref{prop:lin_dep} allows us to determine all the linear dependence relations among the $f_i$, in some cases it may be easier to use the following condition.
Roughly speaking, it says that if $\mathcal{A}$ contains subautomata having the same structure and their output vectors are linearly dependent, then the same linear dependence carries over to the finite state functions corresponding to the states of these subautomata. To simplify the notation write $f_q = f_i$, when $q=q_i$.

\begin{prop} \label{prop:isomorphism}
Let $Q' \subset Q$ be nonempty and such that $\delta(Q' \times \Sigma_k^*) \subset Q'$. Assume that a function $\rho\colon Q' \rightarrow Q'$  satisfies $\delta(\rho(q),j)= \rho(\delta(q,j))$ for each $q \in Q'$ and $j \in \Sigma_k$.
 Let $m \geq 1$ be minimal such that $\rho^m = \rho^j$ for some $j<m$. Choose a state $q' \in Q'$ and let $\beta_0, \ldots, \beta_{m-1} \in \mathbb{C}$. Then
 $$ \sum_{i=0}^{m-1} \beta_i f_{\rho^i(q)} = 0$$
for all $q \in \delta(\{q'\} \times \Sigma_k^*)$ if and only if
$$\sum_{i=0}^{m-1} \beta_i \uptau( \rho^i(q)) = 0$$
for all $q \in \delta(\{q'\} \times \Sigma_k^*)$.
\end{prop}

We illustrate the use of this criterion in the  example below.
\begin{ex}
Consider the $2$-DFAO in Figure \ref{fig:isomorphism}.
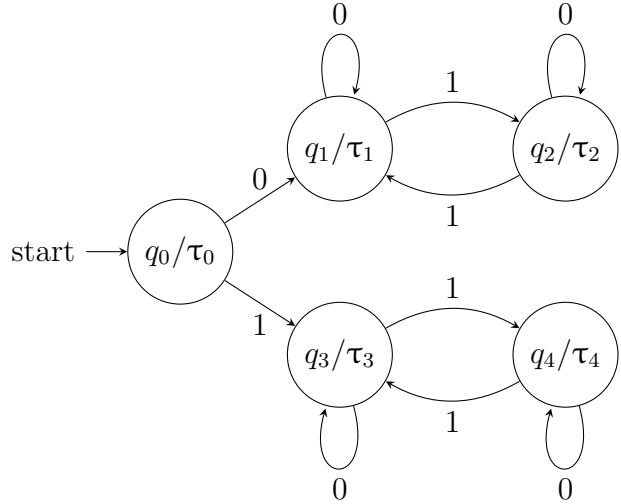
\begin{figure}[ht]
\centering
\begin{tikzpicture}
\node[state, initial](q0){$q_0/\uptau_0$};
\node[state, above right of=q0,yshift=-0.75cm](q1){$q_1/\uptau_1$};
\node[state, right of=q1](q2){$q_2/\uptau_2$};
\node[state, below right of=q0,yshift=0.75cm](q3){$q_3/\uptau_3$};
\node[state,right of=q3](q4){$q_4/\uptau_4$};

\draw
 (q0) edge[above] node{0} (q1)
 (q0) edge[below] node{1} (q3)
 (q1) edge[loop above] node{0} (q1)
 (q1) edge[bend left, above] node{1} (q2)
  (q2) edge[loop above] node{0} (q2)
 (q2) edge[bend left, below] node{1} (q1)
 (q3) edge[loop below] node{0} (q3)
 (q3) edge[bend left, above] node{1} (q4)
  (q4) edge[loop below] node{0} (q4)
 (q4) edge[bend left, below] node{1} (q3);
\end{tikzpicture}
\caption{An example $2$-DFAO}
\label{fig:isomorphism}
\end{figure}
We will show that regardless of the output,  $\dim(\operatorname{span}_{\mathbb{C}}\{f_0,f_1,f_2,f_3,f_4\}) \leq 3$. 
Let $Q' = \{q_1,q_2,q_3,q_4\}$ and define
\[\rho(q_1) = q_3, \qquad \rho(q_3)= q_2, \qquad \rho(q_2)= q_4, \qquad \rho(q_4)=q_1.\]
Then $\rho$ satisfies the assumption in Corollary \ref{prop:isomorphism}.

Choose $q'=q_1$, so that $\delta(\{q'\}\times \Sigma_2^*) = \{q_1,q_2\}$. Proposition \ref{prop:isomorphism} with $m=4$ says that an equality of the form
\[ \beta_0 (\uptau_1, \uptau_2) + \beta_1 (\uptau_3,\uptau_4)+\beta_2(\uptau_2,\uptau_1)+\beta_3(\uptau_4,\uptau_3)=0,\]
where $\beta_0, \beta_1, \beta_2, \beta_3 \in \mathbb{C}$, is equivalent to
\[ \beta_0 (f_1, f_2) + \beta_1 (f_3,f_4)+\beta_2(f_2,f_1)+\beta_3(f_4,f_3)=0.\]
\end{ex}

Proposition \ref{prop:isomorphism} turns out to be useful when studying certain pattern counting sequences.
Let $e_{k:v}(n)$ denote the number of occurrences of a pattern $v \neq \epsilon$ in the base-$k$ expansion of $n$ without leading zeros.

Fix  an integer  $m \geq 2$ and let $\xi_m$ be a primitive $m$th root of unity (we do not assume any relation between $\xi_m$ and $\omega$).
Consider the sequence
\begin{equation} \label{eq:pattern}
a(n) = \xi_m^{e_{k:v}(n)},
\end{equation}
which counts the number of (possibly overlapping) occurrences of the pattern $v$ modulo $m$. Such a sequence $\{a(n)\}_{n \geq 0}$ is $k$-automatic.
In particular, for $k=2, v =\texttt{1},$ and $m=2$ we get the Thue--Morse sequence, whereas $k=2, v =\texttt{11}$ and $m=2$ yields the Rudin--Shapiro sequence. The following result, whose proof utilizes Proposition \ref{prop:isomorphism}, gives a sharper bound on $l_{\min}$ for this class of automatic sequences.

\begin{prop} \label{prop:pattern}
Let $\{a(n)\}_{n \geq 0}$ be defined as in \eqref{eq:pattern}.
If $v$ has no leading zeros, then $l_{\min} \leq |v|$; otherwise,  $l_{\min} \leq |v|+1$. 
\end{prop}

\section{The integrality of coefficients} \label{sec:Integrality}
In this section we investigate the integrality of coefficients of the recurrence relation (\ref{eq:recurrence_form}) for a general $k$-automatic sequence $\{a(n)\}_{n \geq 0}$ under some mild assumptions. 
Let $C(x,y) \in \mathbb{C}[x,y]$  be the characteristic polynomial (in $y$) either of  $\widehat{M}(k^s;x)$, in the case when $\mathcal{A}$ forward-induces  $\{a(n)\}_{n \geq 0}$, or of $\widehat{M}^R(k^s;x)$, when $\mathcal{A}$ backward-induces  $\{a(n)\}_{n \geq 0}$. Write 
$$C(x,y) = \sum_{m=0}^{c} C_m(x) y^m,$$
where $C_m(x) \in \mathbb{C}[x]$ (recall that $c$ denotes the dimension of $\widehat{M}(x)$).  As we have mentioned earlier, the recurrence relation \eqref{eq:recurrence_form} holds with the numbers $C_m(\omega)$ playing the role of the coefficients. 
Let $r_0 \geq 1$ be minimal such that $\omega^{r_0}=1$ and let $s_0$ denote the multiplicative order of $k$ modulo $r_0$. Let $\psi_k$ be the automorphsim of $\mathbb{Q}(\omega)$ mapping $\omega$ to $\omega^k$.
It turns out that if the entries of $\widehat{M}(x)$ are polynomials with rational coefficients, then we can explicitly indicate a subfield of $\mathbb{Q}(\omega)$ of dimension $\varphi(r_0)/s_0$ over $\mathbb{Q}$, containing all $C_m(\omega)$. The second part of Proposition \ref{prop:real} is a special case of this result.
\begin{thm} \label{thm:extension}
If $\widehat{m}_{ij} \in \mathbb{Q}[x]$ for all $i,j \in \{0,1,\ldots,c\}$, then the elements $C_m(\omega)$ lie in the subfield of $L \subset \mathbb{Q}(\omega)$ fixed by the subgroup $\langle \psi_k \rangle$ generated by $\psi_k$.
\end{thm}
The assumption $\widehat{m}_{ij} \in \mathbb{Q}[x]$ is satisfied whenever $\widehat{M}(x)$ corresponds to a set of generators of $\operatorname{span}_{\mathbb{Q}}\{f_0,\ldots,f_{d-1}\}$, in particular if we consider$\widehat{M}(x) = M(x)$. 
In the case of $r_0$ is squarefree  we can give a more explicit description of $L$. Let $f = \varphi(r_0)/s_0$ and consider the \emph{Gaussian periods}  $\eta_0, \ldots, \eta_{f-1}$, defined
\[\eta_j = \sum_{i=0}^{s_0-1}\omega^{jk^i}.\]
Then $\eta_0, \ldots, \eta_{f-1}$ form an integral basis of $L$ over $\mathbb{Q}$. It is also easy to show that in this case $L = \mathbb{Q}(\eta_0)$. Indeed, we have $ \mathbb{Q}(\eta_0) \subset L$. To prove the reverse inclusion, observe that $\eta_0, \ldots, \eta_{f-1}$ are all the distinct Galois conjugates of $\eta_0$. Hence, $[\mathbb{Q}(\eta_0)\colon\mathbb{Q}] = f$, which proves our claim.
The following immediate corollary of Theorem \ref{thm:extension}  is a partial generalization of Proposition \ref{prop:prime_power}(i) for an arbitrary $k$-automatic sequence.
\begin{cor} \label{cor:primitive}
Assume that $r_0$ is a prime power and $k$ is a primitive root modulo $r_0$. If $\widehat{m}_{ij} \in \mathbb{Q}[x]$
for all $i,j \in \{0,1,\ldots,c\}$, then $C_m(\omega) \in \mathbb{Q}$ 
 for $m = 0,1,\ldots,c$.
\end{cor}
In general, we cannot expect $C_m(\omega)$ to be integers unless $\widehat{m}_{ij} \in \mathbb{Z}[x]$ (in such case $C_m(\omega)$ is both rational and an algebraic integer, hence an integer). This is not a concern if  all $C_m(\omega)$ are rational , since we can multiply the recurrence by an appropriate integer to clear their denominators.
Even if the numbers $C_m(\omega)$ are not rational, it is possible to obtain a recurrence relation of desired form with integer coefficients. 

\begin{thm} \label{thm:int_coef}
Let $q = |(\mathbb{Z}/r_0\mathbb{Z})^{\times}/\langle k \rangle|$ and choose $u_0 = 1, u_1, \ldots, u_{q-1}$ representatives of distinct cosets in $(\mathbb{Z}/r_0\mathbb{Z})^{\times}/\langle k \rangle$. Write
\[\prod_{j=0}^{q-1} C(x^{u_j},y) = \sum_{m=0}^{qc} D_m(x) y^m.\]
If $\widehat{m}_{ij} \in \mathbb{Q}[x]$ for all $i,j \in \{0,1,\ldots,c\}$, then for all $n \geq 1$ we have
\[\sum_{m=0}^{qc} D_m(\omega) A(k^{ms}n;\omega) = 0\]
and $D_m(\omega) \in \mathbb{Q}$ for $m=0,\ldots,ql$.
\end{thm}

\section{Proofs} \label{sec:Proofs}
In this section we present the proofs of the results in this paper, along with some auxiliary lemmas. In each case we retain the notation from the corresponding section.
\subsection{Proofs of results in Section \ref{sec:TM}}
\begin{proof}[Proof of Proposition \ref{prop:TM_props}]
The identity (i) is an immediate consequence of \eqref{eq:TM_rec}, and further implies (ii). Part (iii) is \cite[Lemma 8.1]{DM00} (keep in mind the shift in indexing the polynomials), but can also be obtained straight from (i) and (ii). Equality (iv) is obviously true for $s=0$. By (ii) and induction on $s$  we get
\begin{align*}
x^{2^{s+1}-1} T\left(2^{s+1}; \frac{1}{x}\right) &=x^{2^{s+1}-1} (1-x^{-2^s})  T\left(2^{s}; \frac{1}{x}\right) =  (-1)^s(x^{2^s}-1) T(2^{s};x) \\
&=  (-1)^{s+1} T(2^{s+1}; x). \qedhere
\end{align*}
\end{proof}

Before proceeding to the further proofs we fix some additional notation. Let $\varphi$ denote the Euler totient function and let $\Phi_{n}$ be the $n$th cyclotomic polynomial. Let $ \psi_m \in \operatorname{Gal} (\mathbb{Q}(\omega)/\mathbb{Q})$ denote the automorphism taking $\omega$ to $\omega^m$, where $m \in \mathbb{Z}$ is coprime to $r_0$. In particular, we write $\overline{z} = \psi_{-1}(z)$ for  complex conjugation. The multiplicative group $(\mathbb{Z}/r_0\mathbb{Z})^{\times}$ and $\operatorname{Gal} (\mathbb{Q}(\omega)/\mathbb{Q})$ are isomorphic and of order $\varphi(r_0)$.
Let $\langle 2 \rangle $ denote the cyclic subgroup of $(\mathbb{Z}/r_0\mathbb{Z})^{\times}$ generated by $2$. 

\begin{proof}[Proof of Proposition \ref{prop:real}]
We have
$$\overline{T(2^{s_0};\omega)} = T(2^{s_0};\omega^{-1}) = (-1)^{s_0} T(2^{s_0}; \omega),$$
where we used Proposition \ref{prop:TM_props}(iv). The first part of the claim follows immediately.

To prove the second part we observe that $T(2^{s_0};\omega^2) = T(2^(s_0);\omega)$, hence this number is invariant under the action of the subgroup of $\operatorname{Gal} (\mathbb{Q}(\omega)/\mathbb{Q})$ generated by $\psi_2$. This subgroup is of order $s_0$ and the result follows from the fundamental theorem of Galois theory. 
\end{proof}

The following auxiliary lemma establishes a relation between $T(2^{s_0};\omega)$ and $\Phi_{r_0}(1)$.
\begin{lem} \label{lem:conj}
Let $q = |(\mathbb{Z}/r_0\mathbb{Z})^{\times}/\langle 2 \rangle|$ and assume that $1, m_1, \ldots, m_{q-1}$ are representatives of distinct cosets in $(\mathbb{Z}/r_0\mathbb{Z})^{\times}/\langle 2 \rangle$. Then
$$\Phi_{r_0}(1) = T(2^{s_0};\omega) \prod_{j=1}^{q-1} \psi_{m_j} (T(2^{s_0};\omega)).$$
\end{lem}
\begin{proof}
Put $m_0 = 1$. Then we have
$$\Phi_{r_0}(1) = \prod_{\underset{(r_0,m)=1}{1 \leq m \leq r_0-1}}(1-\omega^m) = \prod_{j=0}^{q-1}\prod_{m \in \langle 2 \rangle}  (1-\omega^{m m_j}) = T(2^{s_0};\omega) \prod_{j=1}^{q-1} \psi_{m_j} (T(2^{s_0};\omega)).  \qedhere$$
\end{proof}

\begin{proof}[Proof of Proposition \ref{prop:prime_power}]
If $s_0 = \varphi(r_0)$, then we have $q=1$ in Lemma \ref{lem:conj} and it follows that $T(2^{s_0};\omega) = p$.
To prove (i) it remains to show that if $T(2^{s_0};\omega) \in \mathbb{Z}$, then $2$ is a primitive root modulo $r_0$. Suppose, on the contrary, that $\varphi(r_0)=s_0 q$ for some $q >1$. The value $T(2^{s_0};\omega)$ is fixed under the action of $\operatorname{Gal} (\mathbb{Q}(\omega)/\mathbb{Q})$. Again, by Lemma \ref{lem:conj}, we get 
\[
p = (T(2^{s_0};\omega))^q,\]
thus a contradiction.

To prove (ii), assume that $s_0 = \varphi(r_0)/2$ is odd. This means that $- 1 \not\in \langle 2 \rangle$, thus
$$p = \Phi_{r_0}(1) = T(2^{s_0};\omega) T(2^{s_0};\omega^{-1}) = |T(2^{s_0};\omega)|^2,$$
and the result follows, since $T(2^{s_0};\omega)$ is purely imaginary by Proposition \ref{prop:real}.
\end{proof}


\begin{proof}[Proof of Proposition \ref{prop:distinct_factors}]
Lemma \ref{lem:conj} implies that $T(2^{s_0};\omega)$ is a unit in the ring of integers of $\mathbb{Q}(\omega)$. The only such rational integers are $1,-1$, which proves (i). 

Under the assumptions of (ii), $s_0$ is even and Proposition \ref{prop:real} gives $T(2^{s_0};\omega) \in \mathbb{R}$. Since $2^{s_0/2} \not\equiv -1 \pmod{r_0}$, we obtain $-1 \not\in \langle 2 \rangle$. As in the proof of Proposition  \ref{prop:prime_power}(ii), we get
$$1 = \Phi_{r_0}(1) =  |T(2^{s_0};\omega)|^2,$$
and the result follows.

In order to prove (iii) we observe that $2^{s_0/2} \equiv -1 \pmod{r_0}$ gives
$$ T(2^{s_0};\omega) = \prod_{j=0}^{s_0/2-1} (1-\omega^{2^j})  \prod_{j=0}^{s_0/2-1}(1-\omega^{-2^j}) = |T(2^{s_0/2};\omega)|^2.$$
Suppose that $T(2^{s_0};\omega) \in \mathbb{Z}$. Then we must have $T(2^{s_0};\omega) = 1$ by Proposition \ref{prop:distinct_factors}. However,
$$\overline{T(2^{s_0/2};\omega)} =  (-1)^{s_0/2} \omega^{-2^{s_0/2}+1} T(2^{s_0/2}; \omega) = (-1)^{s_0/2} \omega^2 T(2^{s_0/2};\omega),$$
by equality (iv) of Proposition \ref{prop:TM_props}, which means that
$$(-1)^{s_0/2} = [ \omega T(2^{s_0/2},\omega)]^2.$$
If $s_0 \equiv 2 \pmod{4}$ this immediately leads to a contradiction, since $\mathbb{Q}(\omega)$ does not contain a square root of $-1$. If $s_0 \equiv 0 \pmod{4}$ we obtain
$$(\omega T(2^{s_0/2};\omega)-1)(\omega T(2^{s_0/2};\omega)+1)=0.$$
This means that the value $\omega T(2^{s_0/2};\omega)$ is invariant under the action of $\operatorname{Gal}(\mathbb{Q}(\omega)/\mathbb{Q})$ and equals either $1$ or $-1$. However, 
$$ \psi_2(\omega T(2^{s_0/2};\omega)) = \omega^2 T(2^{s_0/2};\omega) \frac{1-\omega^{2^{s_0/2}}}{1-\omega} = -\omega T(2^{s_0/2};\omega),$$
where we used $2^{s_0/2} \equiv -1 \pmod{r_0}$.
\end{proof}

\subsection{Proofs of results in Section \ref{sec:Matrix}}

\begin{proof}[Proof of Proposition \ref{prop:matrix_coef}]
The formula (\ref{eq:matrix_coef_1})  holds for $t=0$. For $t \geq 0$  by induction we have
\begin{align*}
 m_{ij}(k^{t+1};x) &= \sum_{l=0}^{d-1} m_{il}(x^{k^t}) m_{lj}(k^t;x) = \sum_{l=0}^{d-1} \sum_{\underset{ \delta(q_i , a) = q_l}{a \in \Sigma_k}} x^{ak^t}  \sum_{\underset{\delta(q_l , w)=q_j}{w \in \Sigma_k^t}} x^{[w]_k} \\
 &=  \sum_{\underset{ \delta(q_i , aw) = q_l}{a \in \Sigma_k, w \in \Sigma_k^t}} x^{[aw]_k}  = \sum_{\underset{\delta(q_i , w)=q_j}{w \in \Sigma_k^{t+1}}} x^{[w]_k}.
 \end{align*}
The proof of (\ref{eq:matrix_coef_2}) is analogous.
\end{proof}

The following auxilliary result shows that the map $v \mapsto M_v$ is a homomorphism of monoids.
\begin{lem} \label{lem:matrix_word}
Let $v,w \in \Sigma_k^{*}$. Then
$$
M_{vw} = M_v M_w.
$$
\end{lem}
\begin{proof}
Fix $i,j$ and let $q_l = \delta(q_i,v)$, that is, $m_{v,il}=1$. Then by definition $m_{vw, ij} = 1$ if and only if  $\delta(q_l ,w) = \delta(q_i,vw)  = q_j$. This is further equivalent to $m_{w,lj} = 1$, which proves our claim.
\end{proof}

\begin{proof}[Proof of Proposition \ref{prop:matrix_poly}]
The equality (\ref{eq:matrix_poly_1}) is true for $t=0$.  For $t \geq 0$  by induction we have
$$M(k^{t+1};x)  = \sum_{a \in \Sigma_k} M_a x^{ak^t} \sum_{w \in  \Sigma_k^t} M_w x^{[w]_k} = \sum_{a \in \Sigma_k, w \in \Sigma_k^t} M_{aw} x^{[aw]_k} = \sum_{w \in \Sigma_k^{t+1}} M_w x^{[w]_k},$$
where in the second equality we used Lemma \ref{lem:matrix_word}. The proof of (\ref{eq:matrix_poly_2}) is analogous.
\end{proof}

\subsection{Proofs of results in Section \ref{sec:Recurrence}}

Let $f(w)= [f_0(w), \ldots,f_{d-1}(w) ]^T$ for $w \in \Sigma_k^*$. We first state an easy observation, which can be proved similarly as Lemma \ref{lem:matrix_word}.
\begin{lem} \label{lem:concat}
For all $v, w \in \Sigma_k^{*}$, there holds 
\begin{equation*} 
 f(vw) = M_v f(w)
\end{equation*}
\end{lem}

\begin{proof}[Proof of Proposition \ref{prop:poly_recur}]
 Denote for $n \geq 1$
 \[F(n;x)= [F_0(n;x), \ldots,F_{d-1}(n;x) ]^T.\]
First, we prove that our claim holds for $n=k^t$ (in fact, we only need $n=k$), and with $\widehat{F}, \widehat{M}$ replaced by $F, M$. We have
\begin{align}  \label{eq:recur_left_2}
F(k^{u+t};x) &= \sum_{w \in \Sigma_k^{u+t}} f(w) x^{[w]_k} = \sum_{w \in \Sigma_k^t} \sum_{v \in \Sigma_k^u} f(wv) x^{[wv]_k} \nonumber \\
&= \sum_{w \in \Sigma_k^{t}} M_w x^{k^u [w]_k} \sum_{v \in \Sigma_k^u}  f(v) x^{[v]_k} =    M(k^t;x^{k^u}) F(k^u;x),
\end{align}
where we used Lemma \ref{lem:concat} and Proposition \ref{prop:matrix_poly}.
Similarly,
\begin{align} \label{eq:recur_right_2}
F^R(k^{u+t};x) &=  \sum_{w \in \Sigma_k^{u+t}} f(w^R) x^{[w]_k} = 
\sum_{v \in \Sigma_k^u}\sum_{w \in \Sigma_k^{t}}f(v^Rw^R) x^{[wv]_k} \nonumber \\
&= \sum_{v \in \Sigma_k^u} M_{v^R} x^{[v]_k} \sum_{w \in \Sigma_k^{t}} f(w^R) x^{k^u[w]_k}
= M^R(k^u;x) F^R(k^t;x^{k^u}). 
\end{align}

Choose $i \in \{0,1,\ldots,c\}$. Putting $t=1$ in \eqref{eq:recur_left_2} and using \eqref{eq:linear_dep}, we obtain 
\begin{align*}
F_i(k^{u+1};x) &= \sum_{j=0}^{d-1}m_{ij}(x^{k^u})F_j(k^u;x) \\
&= \sum_{j=0}^{c}m_{ij}(x^{k^u})F_j(k^u;x) + \sum_{p=c+1}^{d-1}m_{ip}(x^{k^u})\sum_{j=0}^{c}\alpha_{pj}F_j(k^u;x) \\
&= \sum_{j=0}^{c} \widehat{m}_{ij}(x^{k^u})F_j(k^u;x),
\end{align*}
and thus $\widehat{F}(k^{u+1};x)= \widehat{M}(x^{k^u})\widehat{F}(k^{u};x)$. By induction, for any $t \geq 1$ we get
\begin{equation}\label{eq:recur_left_3}
\widehat{F}(k^{u+t};x)= \widehat{M}(k^t;x^{k^u})\widehat{F}(k^{u};x).
\end{equation}
Now, take any  $n\geq1$ and $t$ such that $k^{t-1}+1 \leq n \leq k^t$.
By truncating the terms of degree $\geq k^un$ in \eqref{eq:recur_left_3} we obtain \eqref{eq:recur_left}. The identity \eqref{eq:recur_right} can be proved by an analogous reasoning, starting from \eqref{eq:recur_right_2}.
\end{proof}

\begin{proof}[Proof of Theorem \ref{thm:recur}]
We consider case 1 first. 
By the assumption that $\mathcal{A}$ forward-induces $\{a(n)\}_{n \geq 0}$  we have $A(n;x) = F_0(n;x)$ for all $n \geq 1$.
The sequence $\{\omega^{k^n}\}_{n \geq 0}$ is periodic with period $s$. 
Hence, $\widehat{M}(k^{ms}; \omega) = \widehat{M}^m(k^s; \omega)$ and by (\ref{eq:recur_left}) we get
$$
\widehat{F}(k^{ms};\omega) =
\widehat{M}^m(k^s; \omega)
\widehat{F}(1;\omega).
$$
Using the assumption $C(\omega, \widehat{M}(k^s;\omega)) = 0$, we obtain 
\begin{equation} \label{eq:recurrence_1}
\sum_{m=0}^{l} C_m(\omega) 
\widehat{F}(k^{ms};\omega) =
0.
\end{equation}
Left-multiplying (\ref{eq:recurrence_1}) by $\widehat{M}(n; \omega^{k^u})$ and using \eqref{eq:recur_left}, we get
$$
\sum_{m=0}^{l} C_m(\omega) 
\widehat{F}(k^{ms}n;\omega) =
0.
$$
The result follows by considering the first entry in this vector.
 
 Case (ii) is slightly easier to prove. By the assumption that $\mathcal{A}$ backward-induces $\{a(n)\}_{n \geq 0}$  we have $A(n;x) = F_0^R(n;x)$ for all $n \geq 1$.
Similarly as before we get $\widehat{M}^R(k^{ms}; \omega) =(\widehat{M}^R(k^s; \omega))^m$. As a consequence, by (\ref{eq:recur_right}) we have
$$\widehat{F}^R(k^{ms}n;\omega) = (\widehat{M}^R(k^s; \omega))^m \widehat{F}^R(n;\omega)$$
for all $n \geq 1$. 
From the assumption
$C(\omega, \widehat{M}^R(k^s;\omega)) = 0$
we obtain 
$$
\sum_{m=0}^{l} C_m(\omega) 
\widehat{F}^R(k^{ms}n;\omega) = 0.
$$
and the result follows.
\end{proof}

\subsection{Proofs of results in Section \ref{sec:Order}}
\begin{proof}[Proof of Proposition \ref{prop:lin_dep}]
For any $w \in \Sigma_k$ there exists $j \in \{0,\ldots,e\}$  such that $\delta(q_i,w)=  \delta(q_i,w_j)$ for all $i=0,\ldots,d-1$. The result follows immediately.
\end{proof}

\begin{proof}[Proof of Proposition \ref{prop:isomorphism}]
Let $q \in \delta(\{q'\} \times \Sigma_k^*)$ and $v \in \Sigma_k^*$. We have
$$\sum_{i=0}\beta_i f_{\rho^i(q)}(v) =  \sum_{i=0}^{m-1} \beta_i \uptau(\delta(\rho^i(q),v)) = \sum_{i=0}^{m-1} \beta_i \uptau(\rho^i(\delta(q,v))),$$
and the result follows, since $\delta(q,v) \in \delta(\{q'\} \times \Sigma_k^*)$.
\end{proof}

\begin{proof}[Proof of Proposition \ref{prop:pattern}]
Let $v =  v_1 \cdots v_e$ with $v_i \in \Sigma_k$ and put $v_0 = \epsilon$ for a more consistent description. 
First, we construct a $k$-DFAO $\mathcal{A} =(Q,\Sigma_k,\delta,q_0,\Delta,\uptau)$ which returns $a(n)$ given the input $(n)_k$. 
Put $Q = \{0,1,\ldots,m-1\} \times \{0,1,\ldots,e-1\}$ and $q_0 = (0,0)$. We define the transition function $\delta$ in such a way, that arriving at a state $(p,t) \in Q$ means the following:
\begin{itemize}
\item counting modulo $m$, so far $p$ occurences of $v$ in $(n)_k$ have been found,
\item $t$ is the largest number such that the part of the input read so far is of the form $w v_0 \cdots v_t$ for some $w \in \Sigma_k^{*}$. 
\end{itemize}
More precisely, fix $j \in \Sigma_k$, the current symbol in the input.  Let $t'$ be the length of the longest suffix of $v_0 v_1 \cdots v_{t} j$ which is also a proper prefix of $v$ (keeping in mind that $|v_0| = 0$). Roughly speaking, this means that after reading $j$, we will have read $t'$ symbols of the next (potential) occurrence of $v$. Now we consider two possibilities.
 If $t=e-1$ and $j = v_e$, then we define $\delta((p,e-1),v_e) = (p+1 \bmod{m},t')$. Roughly speaking, this means that the symbol currently being read successfully completes an occurrence of $v$ and we start counting again towards the next occurrence (the definition of $t'$ takes into account the possiblility of overlapping ocurrences of $v$). 
In all the other cases we let $\delta((p,t),j) = (p,t')$, in particular $\delta((p,t),v_{t+1}) = (p,t+1)$ for $t < e-1$.
 Finally, let $\uptau(p,t) = \xi_m^{p}$. It is clear from the interpretation of the states $(p,t)$ that indeed $\uptau(\delta(q_0, (n)_k)) = a(n)$. 

If $v_1 \neq 0$, then $\delta(q_0,0)=q_0$, thus $\mathcal{A}$ forward-induces $\{a(n)\}_{n \geq 0}$. Otherwise, recall from Section \ref{sec:Preliminaries} that we can add to $\mathcal{A}$ a new initial state $q_0'$ to create a $k$-DFAO $\mathcal{A}'$, in which the structure of $\mathcal{A}$ is preserved and which forward-induces $\{a(n)\}_{n \geq 0}$.

Now, in order to use Prop \ref{prop:isomorphism}, define $\rho \colon Q \to Q$ by $$\rho(p,t) = (p+1 \bmod{m}, t).$$ Then clearly $\delta(\rho(q),j) = \rho(\delta(q,j))$ and $\uptau (\rho(q))= \xi_m \uptau (\rho(q)) $ for all $q \in Q$ and $j \in \Sigma_{k}$. By Proposition \ref{prop:isomorphism} it follows that for each $q \in Q$
$$f_{\rho(q)} = \xi_m f_q,$$
where $f_q$ denotes the finite-state function corresponding to $q$. Since $\rho$ has exactly $e= |v|$ orbits, there are at most $|v|$  linearly independent finite state functions $f_q$ for $q \in Q$.  If $v$ begins with $0$, we also need to account for the finite-state function corresponding to $q'_0$. This ends the proof by Proposition \ref{prop:lin_dep}.
\end{proof}

\subsection{Proofs of results in Section \ref{sec:Integrality}}

\begin{proof}[Proof of Theorem \ref{thm:extension}]
Assume that $\mathcal{A}$ forward-induces $\{a(n)\}_{n \geq 0}$ (the other case is proved similarly).
Observe that
 \begin{align*}
\widehat{M}(k^s; \omega) &=\left( \widehat{M}(\omega^{k^{n-1}}) \cdots  \widehat{M}(\omega^{k^2})  \widehat{M}(\omega^k) \right) \widehat{M}(\omega)  , \\
\widehat{M}(k^s; \omega^k) &= \widehat{M}(\omega) \left( \widehat{M}(\omega^{k^{n-1}}) \cdots  \widehat{M}(\omega^{k^2})  \widehat{M}(\omega^k)\right),
 \end{align*}
 which implies that $\widehat{M}(k^s; \omega^k)$ and $\widehat{M}(k^s; \omega)$ have the same characteristic polynomial $C(\omega,y) \in (\mathbb{Q}(\omega))[y]$. It follows that $\psi_k(C_m(\omega)) = C_m(\omega)$,  thus $ C_m(\omega)$ is invariant under the action of $\langle \psi_k \rangle$. 
\end{proof}

\begin{proof}[Proof of Theorem \ref{thm:int_coef}]
Since $C(x,\widehat{M}(k^s;\omega))=0$, Theorem \ref{thm:recur} implies that for all $n \geq 1$
\[\sum_{m=0}^{qc} D_m(\omega) A(k^{ms}n;\omega) = 0.\]
By Theorem  \ref{thm:extension}, $C_j(\omega), C_j(\omega^{u_1}),\ldots,C_j(\omega^{u_{q-1}})$ are all the Galois conjugates  (not necessarily distinct) of $C_j(\omega)$ for each $j=0,\ldots,c$.  The coefficients $D_m(\omega)$ are symmetric polynomials in $C_0(\omega), \ldots, C_c(\omega)$, and therefore $D_m(\omega) \in \mathbb{Q}$.
\end{proof}

\bibliographystyle{amsplain}
\bibliography{references}

\providecommand{\bysame}{\leavevmode\hbox to3em{\hrulefill}\thinspace}
\providecommand{\MR}{\relax\ifhmode\unskip\space\fi MR }
\providecommand{\MRhref}[2]{%
  \href{http://www.ams.org/mathscinet-getitem?mr=#1}{#2}
}
\providecommand{\href}[2]{#2}
\begin{thebibliography}{1}

\bibitem{AS03}
Jean-Paul Allouche and Jeffrey Shallit, \emph{Automatic sequences}, Cambridge
  University Press, Cambridge, 2003, Theory, applications, generalizations.
  \MR{1997038}

\bibitem{BLM76}
John Brillhart, J.~S. Lomont, and Patrick Morton, \emph{Cyclotomic properties
  of the {R}udin-{S}hapiro polynomials}, J. Reine Angew. Math. \textbf{288}
  (1976), 37--65. \MR{0498479}

\bibitem{Do01}
Christophe Doche, \emph{On the real roots of generalized {T}hue-{M}orse
  polynomials}, Acta Arith. \textbf{99} (2001), no.~4, 309--319. \MR{1845687}

\bibitem{DM00}
Christophe Doche and Michel~Mend\`es France, \emph{Integral geometry and real
  zeros of {T}hue-{M}orse polynomials}, Experiment. Math. \textbf{9} (2000),
  no.~3, 339--350. \MR{1795306}

\bibitem{Ko97}
Dexter~C. Kozen, \emph{Automata and computability}, Undergraduate Texts in
  Computer Science, Springer-Verlag, New York, 1997. \MR{1633052}

\bibitem{Ru59}
Walter Rudin, \emph{Some theorems on {F}ourier coefficients}, Proc. Amer. Math.
  Soc. \textbf{10} (1959), 855--859. \MR{116184}

\bibitem{Sh53}
Harold~S. Shapiro, \emph{Extremal problems for polynomials and power series},
  ProQuest LLC, Ann Arbor, MI, 1953, Thesis (Ph.D.)--Massachusetts Institute of
  Technology. \MR{2938495}

\end{thebibliography}
\end{document}